\documentclass{article}%
\usepackage{amsmath}
\usepackage{amsfonts}
\usepackage{amssymb}
\usepackage{graphicx}%
\providecommand{\U}[1]{\protect\rule{.1in}{.1in}}
\newtheorem{theorem}{Theorem}[section]

\newtheorem{corollary}[theorem]{Corollary}

\newtheorem{example}[theorem]{Example}

\newtheorem{lemma}[theorem]{Lemma}

\newtheorem{proposition}[theorem]{Proposition}

\newenvironment{proof}[1][Proof]{\noindent\textbf{#1.} }{\ \rule{0.5em}{0.5em}}
\begin{document}

\title{The Borel-Weil theorem for reductive Lie groups }
\author{Jos\'{e} Araujo\\Facultad de Ciencias Exactas, UNICEN \\Tandil, Argentina.
\and Tim Bratten\\Facultad de Ciencias Exactas, UNICEN\\Tandil, Argentina. }
\date{}
\maketitle

\begin{abstract}
In this manuscript we consider the extent to which an irreducible
representation for a reductive Lie group can be realized as the sheaf
cohomolgy of an equivariant holomorphic line bundle defined on an open
invariant submanifold of a complex flag space. Our main result is the
following: suppose $G_{0}$ is a real reductive group of Harish-Chandra class
and let $X$ be the associated full complex flag space. Suppose $\mathcal{O}%
_{\lambda}$ is the sheaf of sections of a $G_{0}$-equivariant holomorphic line
bundle on $X$ whose parameter $\lambda$ (in the usual twisted $\mathcal{D}%
$-module context) is antidominant and regular. Let $S\subseteq X$ be a $G_{0}%
$-orbit and suppose $U\supseteq S$ is the smallest $G_{0}$-invariant open
submanifold of $X$ that contains $S$. From the analytic localization theory of
Hecht and Taylor one knows that there is a nonegative integer $q$ such that
the compactly supported sheaf cohomology groups $H_{\text{c}}^{p}%
(S,\mathcal{O}_{\lambda}\mid_{S})$ vanish except in degree $q$, in which case
$H_{\text{c}}^{q}(S,\mathcal{O}_{\lambda}\mid_{S})$ is the minimal
globalization of an associated standard Beilinson-Bernstein module. In this
study we show that the $q$-th compactly supported cohomolgy group
$H_{\text{c}}^{q}(U,\mathcal{O}_{\lambda}\mid_{U})$ defines, in a natural way,
a nonzero submodule of $H_{\text{c}}^{q}(S,\mathcal{O}_{\lambda}\mid_{S})$,
which is irreducible (i.e. realizes the unique irreducible submodule of
$H_{\text{c}}^{q}(S,\mathcal{O}_{\lambda}\mid_{S})$) when an associated
algebraic variety is nonsingular. By a tensoring argument, we can show that
the result holds, more generally (for the nonsingular case), when the
representation $H_{\text{c}}^{q}(S,\mathcal{O}_{\lambda}\mid_{S})$ is what we
call a classifying module.

\end{abstract}

\section{Introduction}

In this manuscript we show there is a natural generalization of the Borel-Weil
theorem to the class of reductive Lie groups which serves to realize many, but
not all, irreducible admissible representations.

Starting with Schmid's thesis \cite{schmid1}, there are general results
realizing irreducible representations as sheaf cohomologies of finite-rank
holomorphic vector bundles defined over open orbits in generalized complex
flag spaces \cite{wong}, \cite{bratten1}. However, relatively few irreducible
representations can be realized this way. The equivariant $\mathcal{D}$-module
theory of Beilinson and Bernstein \cite{beilinson} provides a powerful
generalization to the Borel-Weil theorem and produces geometric realizations
for any irreducible Harish-Chandra module. However, one would also like to
find a natural realization of a corresponding group representation. In a
general sense, the analytic localization defined by Hecht and Taylor
\cite{hecht} does just that, by giving realizations for the minimal
globalizations \cite{schmid2} of Harish-Chandra modules. Along these lines,
the theory of analytic localization was used by Hecht and Taylor to realize
minimal globalizations of the standard modules defined by the
Beilinson-Bernstein theory. In many cases standard modules are irreducible,
but when they are not it is not obvious how to proceed. One difficulty is that
the geometric realization defined by Hecht and Taylor is obtained via an
equivalence of derived categories so that the analytic localization of an
irreducible representation can (and sometimes does) appear as a complex of
sheaves that has nonzero homologies in various degrees. In spite of this
difficulty, it turns out (somewhat surprisingly to us) that the theory of
analytic localization can be used to realize many more irreducible
representations than the example of irreducible standard modules. In the end,
one sees that the key point hinges on whether a certain associated algebraic
variety has singularities. When it does not, then the Beilinson-Bernstein
realization of the corresponding irreducible Harish-Chandra module has a
simple geometric description and this fact controls the analytic localization.
When the associated variety is singular, this simplicity breaks down and it
turns out it is impossible to realize the irreducible representation as the
sheaf cohomology of a finite-rank holomorphic vector bundle defined over an
invariant open submanifold of a generalized flag manifold.

Rather than make a general statement about our main results in the
introduction (our main results are Theorem 5.1 and Corollary 6.2) we would
like to illustrate how the theory works in the context of a connected complex
reductive group where the relationship to the Beilinson-Bernstein
classification of irreducible admissible representations is more transparent.
In particular, suppose $G_{0}$ is a connected complex reductive group with Lie
algebra $\mathfrak{g}_{0}$ and let $K_{0}\subseteq G_{0}$ be a compact real
form. Associated to $K_{0}$ is a corresponding Cartan involution $\theta
:G_{0}\rightarrow G_{0}$ (in this case $\theta$ is the conjugation given by
the real form). Let $X_{0}$ be the complex flag manifold of Borel subgroups of
$G_{0}$ and let $X_{0}^{\text{c}}$ be the conjugate complex manifold. Then the
flag manifold $X$ of Borel subalgebras of the complexified Lie algebra
$\mathfrak{g}$ of $\mathfrak{g}_{0}$ can be identified with the direct
product
\[
X=X_{0}\times X_{0}^{\text{c}}.
\]
We need to consider two actions of $G_{0}$ on $X$. The diagonal action
\[
g\cdot(x,y)=(gx,gy)
\]
corresponds to the fact that $G_{0}$ is a real group with real Lie algebra
$\mathfrak{g}_{0}$ and the action
\[
g\cdot(x,y)=(gx,\theta(g)y)
\]
corresponds to the fact that $G_{0}=K$ is the complexification of $K_{0}$.
Choose a $\theta$-stable Cartan subgroup $H_{0}\subseteq G_{0}$ and a Borel
subgroup $B_{0}\supseteq H_{0}$. Let $W(G_{0})$ be the Weyl group of $H_{0}$
in $G_{0}$. Then we can identify the set of Borel subgroups of $G_{0}$ that
contain $H_{0}$ with $W(G_{0})$ (the identity in $W(G_{0})$ corresponds with
the Borel subgroup $B_{0}$). Let $B_{0}^{\text{op}}$ be the Borel subgroup
opposite to $B_{0}$ with respect to $H_{0}$ (this subgroup corresponds to the
longest element in $W(G_{0})$). Then each $G_{0}$-orbit and each $K$-orbit on
$X$ contains exactly one point of the form
\[
\left(  w\cdot B_{0},B_{0}^{\text{op}}\right)  \in X_{0}\times X_{0}%
^{\text{c}}%
\]
for $w\in W(G_{0})$. Thus the orbits for both actions are simultaneously
parametrized by $W(G_{0})$. Observe that the open orbit for the $G_{0}$-action
and the closed orbit for the $K$-action correspond to the identity in
$W(G_{0})$ We introduce the length function, $l(w)$, on $W(G_{0})$. In
particular, each element $w\in W(G_{0})$ can be expressed as a product of
simple reflections and the corresponding length, $l(w)$, is defined to be the
the number of simple reflections that appear in a minimal expression (i.e. a
reduced word) for $w$. Observe that if $Q_{w}$ is the $K$-orbit corresponding
to $w\in W(G_{0})$ then the complex dimension of $Q_{w}$ is given by
\[
\dim_{\mathbb{C}}(Q_{w})=\dim_{\mathbb{C}}(X_{0})+l(w).
\]

For simplicity we will consider the sheaf of holomorphic functions
$\mathcal{O}_{X}$ on $X$ (more generally one could consider the sheaf of
sections $\mathcal{O}_{\lambda}$ of a $G_{0}$-equivariant holomorphic line
bundle on $X$ whose parameter $\lambda$ in the usual twisted $\mathcal{D}%
$-module context is antidominant and regular). The Beilinson-Bernstein
classification gives a one-to-one correspondence between the equivalence
classes of irreducible admissible representations for $G_{0}$ that have the
same infinitesimal character as the trivial representation and the $G_{0}%
$-orbits on $X$ given in the following way. For $w\in W(G_{0})$ let $S_{w}$ be
the corresponding $G_{0}$-orbit and define
\[
q=\dim_{\boldsymbol{C}}(X_{0})-l(w)\text{.}%
\]
Thus $q$ is the (complex) codimension of the $K$-orbit $Q_{w}$ in $X$. Using
their theory of analytic localization, Hecht and Taylor have shown that the
compactly supported sheaf cohomologies
\[
H_{\text{c}}^{p}(S_{w},\mathcal{O}_{X}\mid_{S_{w}}\mathcal{)}%
\]
of the restriction of $\mathcal{O}_{X}$ to $S_{w}$ vanish except when $p=q$ in
which case $H_{\text{c}}^{q}(S_{w},\mathcal{O}_{X}\mid_{S_{w}}\mathcal{)}$ is
the minimal globalization of a corresponding standard Beilinson-Bernstein
module. It follows that $H_{\text{c}}^{q}(S_{w},\mathcal{O}_{X}\mid_{S_{w}%
}\mathcal{)}$ has a unique irreducible submodule $J_{w}\subseteq H_{\text{c}%
}^{q}(S_{w},\mathcal{O}_{X}\mid_{S_{w}}\mathcal{)}$. These representations
$J_{w}$ for $w\in W(G_{0})$ are exactly the irreducible admissible
representations for $G_{0}$ that have the same infinitesimal character as the
trivial representation. In this manuscript we want to realize the
representations $J_{w}$. Along those lines we introduce the Bruhat order in
$W(G_{0})$ : if $w,u\in W(G_{0})$ then we write $u\preceq w$ if $u$ is an
ordered subword that occurs in a reduced expression for $w$ in terms of
products of simple reflections. Given $w\in W(G_{0}),$ it is well-known that
the Bruhat interval $\Upsilon(w)=\left\{  u\in W(G_{0}):u\leq w\right\}  $
characterizes the Zariski closure of the $K$-orbit $Q_{w}$ in the following
way: \
\[
\overline{Q_{w}}=%
{\displaystyle\bigcup\limits_{u\in \Upsilon(w)}}
Q_{u}.
\]
We call $\overline{Q_{w}}$ \emph{the algebraic variety associated to the
}$G_{0}$-\emph{orbit} $S_{w}$. Define
\[
U_{w}=%
{\displaystyle\bigcup\limits_{u\in \Upsilon(w)}}
S_{u}.
\]
Then $U_{w}$ is the smallest $G_{0}$-invariant open submanifold of $X$ that
contains $S_{w}$ and it is not hard to show that $S_{w}$ is the unique $G_{0}%
$-orbit that is closed in $U_{w}$. Put $U=U_{w}-S_{w}$. Letting $\left(
\mathcal{O}_{X}\mid_{U}\right)  ^{X}$, etc., denote the extension by zero of
the restriction of $\mathcal{O}_{X}$ to $U$, we obtain the following short
exact sequence of sheaves on $X$
\[
0\rightarrow\left(  \mathcal{O}_{X}\mid_{U}\right)  ^{X}\rightarrow\left(
\mathcal{O}_{X}\mid_{U_{w}}\right)  ^{X}\rightarrow\left(  \mathcal{O}_{X}%
\mid_{S_{w}}\right)  ^{X}\rightarrow0.
\]
Using an argument like \cite[Lemma 3.3]{bratten2} it is not hard to show that
\[
H_{\text{c}}^{p}(U,\mathcal{O}_{X}\mid_{U}\mathcal{)}=0\text{ if }p<q+1
\]
and that
\[
H_{\text{c}}^{q}(U_{w},\mathcal{O}_{X}\mid_{U_{w}}\mathcal{)}%
\]
is a nonzero minimal globalization. Thus the long exact sequence in sheaf
cohomology determines an inclusion
\[
H_{\text{c}}^{q}(U_{w},\mathcal{O}_{X}\mid_{U_{w}}\mathcal{)}\rightarrow
H_{\text{c}}^{q}(S_{w},\mathcal{O}_{X}\mid_{S_{w}}\mathcal{)}.
\]
We note that when $U_{w}$ is the preimage of an open $G_{0}$-orbit on a
generalized flag space $Y$ then there is a natural identification of the
representation $H_{\text{c}}^{q}(U_{w},\mathcal{O}_{\mid U_{w}}\mathcal{)}$
with the $q$-th compactly supported cohomology of the holomorphic functions on
the given open orbit in $Y$ (this is one of the key points in \cite{bratten2}%
). When this happens, it is known that the representation $H_{\text{c}}%
^{q}(U_{w},\mathcal{O}_{X}\mid_{U_{w}}\mathcal{)}$ is irreducible. We say that
$U_{w}$ is \emph{parabolic }when $U_{w}$ is the preimage of an open $G_{0}%
$-orbit on a generalized flag space $Y$. Our main result in this study shows
that (more generally) the submodule $H_{\text{c}}^{q}(U_{w},\mathcal{O}%
_{X}\mid_{U_{w}}\mathcal{)}$ is irreducible when the associated algebraic
variety $\overline{Q_{w}}$ is smooth. Thus we can realize $J_{w}$ as
$H_{\text{c}}^{q}(U_{w},\mathcal{O}_{X}\mid_{U_{w}}\mathcal{)}$ when (and, in
fact, only when) this happens. For example when $G_{0}$ is the complex general
linear group $GL(3,\mathbb{C)}$ then 4 out of the 6 $G_{0}$-orbits on $X$ are
parabolic but the algebraic varieties associated to all 6 orbits are smooth so
we can realize all irreducible representations with the given infinitesimal
character in this case. When $G_{0}=GL(4,\mathbb{C})$ then only 8 of the 24
$G_{0}$-orbits are parabolic but 22 out of 24 orbits have smooth associated
varieties so we can realize all but two of the irreducible representations
with the given infinitesimal character (and so on). We will also see (for some
examples) that when the algebraic variety $\overline{Q_{w}}$ is singular then
the representation $H_{\text{c}}^{q}(U_{w},\mathcal{O}_{X}\mid_{U_{w}%
}\mathcal{)}$ is reducible and it is actually impossible to realize the
irreducible representation $J_{w}$ as the compactly sheaf cohomology of an
equivariant (finite-rank) holomorphic vector bundle defined on a $G_{0}%
$-invariant open submanifold in a generalized flag space.

Our manuscript is organized as follows. In the second section we will present
the main results we need about orbits and invariant subspaces in $X$. In the
third section we will introduce the equivariant homogeneous line bundles and
prove the basic embedding theorem. In the fourth section we introduce the
algebraic localization theory and give a geometric description to the
irreducible Harish-Chandra module in the Beilinson-Bernstein classification
assuming the corresponding algebraic variety is smooth. In the fifth section
we introduce the analytic localization and use the comparison theorem to prove
our main result. In the last section we use a tensoring argument to extend our
result to antidominant parameters and also consider how our construction
relates to the classical parabolic induction (in the case of a complex
reductive group) so we can consider some examples. We conclude our manuscript
with a brief consideration of how Serre duality applies. We would like to
mention that the idea of our proof involves a mix of ideas from the two
articles \cite{bratten2} and \cite{bratten3}. Although our argument requires a
heavy use of the $D$-module theory and some familiarity with derived
categories, we would hope it looks natural to anyone familiar with these two
previous articles.

The authors would like to take this moment to acknowledge the importance of
the work by Hecht and Taylor for these results. We would also like to thank
David Vogan and Annegret Paul for being kind enough to help us see how the
representation we are studying fails to be irreducible when the associated
variety is singular.

\section{$G_{0}$-orbits and $K$-orbits}

Throughout this manuscript $G_{0}$ will denote a real reductive Lie group of
Harish-Chandra class with Lie algebra $\mathfrak{g}_{0}$ and complexified Lie
algebra $\mathfrak{g}$. Abusing notation a bit, we let $G$ denote the complex
adjoint group of $\mathfrak{g}$ (note that $G$ has Lie algebra $\left[
\mathfrak{g},\mathfrak{g}\right]  $). There is a natural morphism of Lie
groups
\[
G_{0}\rightarrow G\text{.}%
\]
We also fix a maximal compact subgroup $K_{0}\subseteq G_{0}$ and let $K$ be
the complexification of $K_{0}$. Associated to the maximal compact subgroup
there is an involutive automorphism
\[
\theta:G_{0}\rightarrow G_{0}%
\]
whose fixed point set is $K_{0}$. The involution $\theta$ (as well as the
complexification $\theta:\mathfrak{g}\rightarrow\mathfrak{g}$ of its
derivative) is called \emph{the Cartan involution}.

\emph{A Borel subalgebra of} $\mathfrak{g}$ is a maximal solvable subalgebra.
$G$ acts transitively on the set of Borel subalgebras of $\mathfrak{g}$ and
the resulting homogeneous $G$-space $X$ is a complex projective variety called
\emph{the full flag space of} $G_{0}$. One knows that $G_{0}$ has finitely
many orbits on $X$ \cite{wolf}. Therefore the $G_{0}$-orbits are locally
closed submanifolds.

A basic geometric property of flag manifolds that is fundamental to our study
is the existence of a one-to-one correspondence between $G_{0}$-orbits and
$K$-orbits referred to as \emph{Matsuki duality}. For $x\in X$ we let
$\mathfrak{b}_{x}$ denote the corresponding Borel subalgebra of $\mathfrak{g}%
$. Then \emph{the nilradical $\mathfrak{n}_{x}$ of} $\mathfrak{b}_{x}$ is
given by $\mathfrak{n}_{x}=\left[  \mathfrak{b}_{x},\mathfrak{b}_{x}\right]
$. The point $x\in X$ (as well as the Borel subalgebra $\mathfrak{b}_{x}$) is
called \emph{special }if there is a Cartan subalgebra $\mathfrak{c}$ of
$\mathfrak{b}_{x}$ such that
\[
(i)\text{ }\mathfrak{c}_{0}=\mathfrak{g}_{0}\cap\mathfrak{c}\text{ is a real
form of }\mathfrak{c}\text{ and }(ii)\text{ }\theta(\mathfrak{c}%
)=\mathfrak{c}\text{.}%
\]
Then it has been shown that \cite{matsuki1} both the special points in a
$G_{0}$-orbit as well as the special points in a $K$-orbit form a (non-empty)
$K_{0}$-orbit. A $G_{0}$-orbit $S$ and a $K$-orbit $Q$ are said to be
\emph{Matsuki dual }if $S\cap Q$ contains a special point. It follows that
Matsuki duality defines a bijection between the set of $G_{0}$-orbits and the
set of $K$-orbits. When $S$ is a $G_{0}$-orbit and $Q$ is a $K$-orbit then we
will write $S\sim Q$ when $S$ and $Q$ are Matsuki dual.

Given a $K$-orbit $Q$, \emph{the associated algebraic variety} is defined to
be the Zariski closure $\overline{Q}$ of $Q$ (when $S$ is the $G_{0}$-orbit
dual $Q$ then we will also refer to $\overline{Q}$ as the algebraic variety
associated to $S$). There is a partial order, called \emph{the closure order},
defined on set of $K$-orbits by
\[
Q_{1}\preceq Q_{2}\text{ \ if \ }Q_{1}\subseteq\overline{Q_{2}}\text{.}%
\]
While the associated algebraic variety is a closed $K$-invariant subvariety of
$X$ associated to a $G_{0}$-orbit $S$, we will now define a corresponding
$G_{0}$-invariant open submanifold of $X$. In particular, suppose $S_{0}$ is a
$G_{0}$-orbit and let $Q_{0}$ be the Matsuki dual. We define an associated
indexing set $\Upsilon(S_{0})$ of $G_{0}$-orbits in the following way:
\[
S\in \Upsilon(S_{0})\Longleftrightarrow\exists Q\text{ such that }Q\sim S\text{
and }Q\subseteq\overline{Q_{0}}\text{.}%
\]
We define \emph{the corresponding }$G_{0}$-\emph{invariant subspace }$U$ by
\[
U=%
{\displaystyle\bigcup\limits_{S\in \Upsilon(S_{0})}}
S\text{.}%
\]

\begin{proposition}
Maintain the previously introduced notations. Then $U$ is the smallest $G_{0}%
$-invariant open submanifold that contains $S_{0}$ and $S_{0}$ is the unique
closed $G_{0}$-orbit in $U$.
\end{proposition}

\begin{proof}
We consider the closure orders on the set of $K$-orbits and on the set of
$G_{0}$-orbits. Then Matsuki has shown \cite{matsuki2} that duality reverses
the corresponding closure relations. It follows that
\[
S\in \Upsilon(S_{0})\Longleftrightarrow S_{0}\subseteq\overline{S}%
\Longleftrightarrow S_{0}\preceq S\text{.}%
\]
Thus if $S\notin \Upsilon(S_{0})$ then $S_{0}\cap\overline{S}=\varnothing$ and
therefore $S_{0}\cap\overline{S_{1}}=\varnothing$ for each $G_{0}$-obit
$S_{1}$ contained in $\overline{S}$. Hence
\[
U\cap\overline{S}=\varnothing.
\]
Since there are a finite number of orbits the set
\[
C=%
{\displaystyle\bigcup\limits_{S\notin \Upsilon(S_{0})}}
\overline{S}%
\]
is closed and therefore
\[
U=X-C
\]
is open.

Now suppose $W$ is an open $G_{0}$-invariant submanifold that contains $S_{0}%
$. Suppose $S\subseteq U$. Then $S_{0}\subseteq\overline{S}$. Thus
\[
\overline{S}\cap W\neq\varnothing.
\]
Hence $\overline{S}\cap W$ is a nonempty open $G_{0}$-invariant subset of
$\overline{S}$. Since $S$ is locally closed it follows that $S$ is open and
dense in $\overline{S}$. Hence, from the $G_{0}$-invariance it follows that
\[
S\subseteq\overline{S}\cap W\Longrightarrow S\subseteq W
\]
so that $U\subseteq W$, which proves that $U$ is the smallest $G_{0}%
$-invariant open submanifold that contains $S_{0}$.

To prove the last claim first observe that if $S$ is a $G_{0}$-orbit contained
in $\overline{S_{0}}$ then $S\preceq S_{0}$ so that $S\subseteq U$ if and only
if $S=S_{0}$. Thus%
\[
\overline{S_{0}}\cap U=S_{0}%
\]
and $S_{0}$ is closed in $U$. On the other hand, if $S$ is a closed $G_{0}%
$-orbit in $U$ then
\[
S=\overline{S}\cap U.
\]
However, from the definition of $U$ it follows that for $S\subseteq U$ then
$S_{0}\subseteq\overline{S}$ thus
\[
S_{0}\subseteq\overline{S}\cap U.
\]
It follows that $S=\overline{S}\cap U$ if and only if $S=S_{0}$.
\end{proof}

\begin{example}
Suppose $\mathfrak{p}\subseteq\mathfrak{g}$ is a parabolic subalgebra and let
$Y$ be the corresponding $G$-homogeneous space of parabolic subalgebras of
$\mathfrak{g}$ conjugate to $\mathfrak{p}$. For each $y\in Y$ let
$\mathfrak{p}_{y}$ denote the corresponding parabolic subalgebra of
$\mathfrak{g}$. For $x\in X$ there is a unique $y\in Y$ such that
$\mathfrak{b}_{x}\subseteq\mathfrak{p}_{y}.$Thus there is a canonical
$G$-equivariant projection%
\[
\pi:X\rightarrow Y
\]
given by $\pi(x)=y$ if $\mathfrak{b}_{x}\subseteq\mathfrak{p}_{y}$. A point
$y\in Y$ is called \emph{special }if $\mathfrak{p}_{y}$ contains a special
Borel subalgebra. Suppose $W\subseteq Y$ is an open $G_{0}$-orbit and let
$y\in W$ be a special point. Let $O$ be the $K$-orbit of $y$. Then $O$ is
closed in $Y$ (in fact $O\subseteq W$) and Matsuki has shown \cite{matsuki3}
that the $G_{0}$-orbits in $U=\pi^{-1}(W)$ are Matsuki dual to the $K$-orbits
in $\pi^{-1}(O)$ . One also knows that there is a unique $G_{0}$-orbit $S_{0}$
that is closed in $U$ and that its Matsuki dual $Q_{0}$ is the unique open
orbit in the closed algebraic variety $\pi^{-1}(O)$. Thus $\overline{Q_{0}%
}=\pi^{-1}(O)$ and it follows that a $G_{0}$-orbit $S$ is contained in $U$ if
and only if its Matsuki dual $Q$ is contained in $\overline{Q_{0}}$ thus $U$
is the smallest $G_{0}$-invariant open submanifold that contains $S_{0}$.
Observe that, in this case, the associated algebraic variety $\overline{Q_{0}%
}=\pi^{-1}(O)$ is smooth since the fibers of $\pi$ are smooth and since $\pi$
defines a locally trivial algebraic fibration of $\pi^{-1}(O)$ over $O$.
\end{example}

\section{Equivariant line bundles}

In this section we introduce the equivariant line bundles on the full flag
space $X$, as well as the corresponding standard modules associated to a
$G_{0}$-orbit $S\subseteq X$. We begin this section by introducing the
abstract Cartan dual, which is the parameter set for the the twisted sheaves
of differential operators (TDOs) on $X$. Recall that $G$ is the complex
adjoint group of $\mathfrak{g}$. For $x\in X$ let $\mathfrak{n}_{x}$ the
nilradical of the corresponding Borel subalgebra $\mathfrak{b}_{x}$ and put
\[
\mathfrak{h}_{x}=\mathfrak{b}_{x}/\mathfrak{n}_{x}\text{.}%
\]
Since the stabilizer of $x$ in $G$ (i.e. the corresponding Borel subgroup in
$G$) acts trivially on $\mathfrak{h}_{x}$ it also acts trivially on the
complex dual $\mathfrak{h}_{x}^{\ast}$. It follows that the corresponding
$G$-homogeneous holomorphic vector bundle on $X$ is trivial and that the
associated space of global sections $\mathfrak{h}^{\ast}$ is naturally
isomorphic to $\mathfrak{h}_{x}^{\ast}$ via the evaluation at $x$. The vector
space $\mathfrak{h}^{\ast}$ is called \emph{the abstract Cartan dual for}
$\mathfrak{g}$. If $\mathfrak{c}$ is a Cartan subalgebra of $\mathfrak{b}_{x}%
$, then by coupling the natural projection of $\mathfrak{c}$ onto
$\mathfrak{h}_{x}$ with the evaluation at $x$, we obtain an isomorphism of
$\mathfrak{c}^{\ast}$ with $\mathfrak{h}^{\ast}$ \emph{called the
specialization of} $\mathfrak{h}^{\ast}$ \emph{to} $\mathfrak{c}^{\ast}$
\emph{at} $x$. Using the specializations we can identify an abstract set of
roots
\[
\Sigma\subseteq\mathfrak{h}^{\ast}%
\]
and an abstract set of positive roots
\[
\Sigma^{+}\subseteq\Sigma
\]
where $\Sigma$ corresponds to the set of roots of $\mathfrak{c}$ in
$\mathfrak{g}$ and where $\Sigma^{+}$ corresponds to the roots of
$\mathfrak{c}$ in $\mathfrak{b}_{x}$, via the specialization at $x$. Given
$\alpha\in\Sigma$ and $\mu\in\mathfrak{h}^{\ast}$ we can also define the
complex number
\[
\overset{\vee}{\alpha}(\mu)
\]
\emph{the value of} $\mu$ \emph{on the coroot of }$\alpha$. The element
$\mu\in\mathfrak{h}^{\ast}$ is called \emph{integral} if
\[
\overset{\vee}{\alpha}(\mu)\in\boldsymbol{Z}\text{ \ for each }\alpha\in
\Sigma\text{.}%
\]
It just so happens that the half-sum of positive roots, denoted by $\rho$, is
an integral element of $\mathfrak{h}^{\ast}$ that plays a key role in the TDO parametrization.

Let $\widetilde{G}$ denote the universal cover of $G$ and suppose $\mu
\in\mathfrak{h}^{\ast}$ is integral. For a point $x\in X$, the Lie algebra of
the corresponding Borel subgroup $\widetilde{B_{x}}$ in $\widetilde{G}$ (i.e.
the stabilizer of $x$ in $\widetilde{G}$) is given by $\mathfrak{b}_{x}%
\cap\left[  \mathfrak{g,g}\right]  $. Thus, using the evaluation at $x$, $\mu$
determines a one dimensional representation
\[
\mathfrak{b}_{x}\cap\left[  \mathfrak{g,g}\right]  \rightarrow\mathfrak{h}%
_{x}\overset{\mu_{x}}{\longrightarrow}\mathbb{C}\text{.}%
\]
Since $\widetilde{G}$ is simply connected, it is known that there is a
(unique) holomorphic character of $\widetilde{B_{x}}$ whose derivative is
given, in this way, by $\mu_{x}$. Thus, corresponding to each integral $\mu
\in\mathfrak{h}^{\ast}$ there is a corresponding $\widetilde{G}$-homogeneous
holomorphic line bundle
\[%
\begin{array}
[c]{c}%
\mathbb{L}(\mu)\\
\downarrow\\
X
\end{array}
\cdot
\]
Let $\mathcal{O}(\mu)$ be the corresponding sheaf of holomorphic sections. In
a natural way, $\widetilde{G}$ and thus $\left[  \mathfrak{g,g}\right]  $ act
on $\mathcal{O}(\mu)$. Let $\mathfrak{z}$ be the center of $\mathfrak{g}$.
Suppose $W\subseteq X$ is an open set and let
\[
\sigma:W\rightarrow\mathbb{L}(\mu)
\]
be a local holomorphic section. Then we extend $\mathcal{O}(\mu)$ to a sheaf
of $\mathfrak{g}$-modules by defining
\[
\left(  \xi\cdot\sigma\right)  \left(  x\right)  =\mu(\xi)\sigma(x)\text{ for
}\xi\in\mathfrak{z}\text{ and }x\in W\text{.}%
\]
We say $\mathbb{L}(\mu)$ \emph{is a} $G_{0}$-\emph{equivariant line bundle} if
there exists a $G_{0}$-action on $\mathbb{L}(\mu)$ (in the sense of
differentiable $G_{0}$-actions on vector bundles over differentiable $G_{0}%
$-spaces) such that the induced morphisms $\mathbb{L}(\mu)\rightarrow
\mathbb{L}(\mu)$, given by the multiplication by group elements, are
holomorphic and such that the derivative of the $G_{0}$-action on local
sections coincides with the $\mathfrak{g}$-action.

\begin{example}
An important class of $G_{0}$-equivariant holomorphic line bundles corresponds
to the family of (equivalence classes of) finite-dimensional irreducible
representations of $G_{0}$ that are also irreducible for the corresponding
$\mathfrak{g}$-action (i.e. irreducible finite-dimensional $\mathfrak{g}%
$-modules with a compatible $G_{0}$-action). In particular, let $V$ be a
finite-dimensional $G_{0}$-module that is irreducible as a $\mathfrak{g}%
$-module. For each $x\in X$ let $G_{0}\left[  x\right]  $ be the stabilizer of
$x$ and consider the corresponding $(\mathfrak{b}_{x}$, $G_{0}\left[
x\right]  )$-module given by
\[
V/\mathfrak{n}_{x}V\text{.}%
\]
Choosing a Cartan subalgebra $\mathfrak{c}\subseteq\mathfrak{b}_{x}$ and using
the specialization to $x$, the action of $\mathfrak{c}$ on $V/\mathfrak{n}%
_{x}V$ is given by an element $\mu\in\mathfrak{h}^{\ast}$ that corresponds to
the lowest weight in $V$. Hence, if we define
\[
\lambda=\mu-\rho
\]
then $\overset{\vee}{\alpha}(\lambda)$ is a negative integer, for each
positive root $\alpha\in\Sigma^{+}$. \ We can define the total space of a
$G_{0}$-equivariant holomorphic line bundle $\mathbb{L}_{\lambda}$ by
\[
\mathbb{L}_{\lambda}=%
{\displaystyle\bigcup\limits_{x\in X}}
V/\mathfrak{n}_{x}V\text{.}%
\]
Using the action of $\widetilde{G}$ one can define holomorphic transition
functions. Then the Borel-Weil Theorem says that the representation $V$ is
recovered as the global holomorphic sections of the bundle $\mathbb{L}%
_{\lambda}$.
\end{example}

In general, when $\mathcal{O}(\mu)$ is the sheaf of holomorphic sections of
$G_{0}$-equivariant line bundle we will use the shifted parameter $\lambda
=\mu-\rho$ and write
\[
\mathcal{O}_{\lambda}=\mathcal{O}(\mu)
\]
for the sheaf of holomorphic sections. We say $\lambda$ is \emph{regular} if
\[
\overset{\vee}{\alpha}(\lambda)\neq0\text{ \ for each root }\alpha\in
\Sigma\text{.}%
\]
An element $\lambda$ is called \emph{singular} when it is not regular. We say
$\lambda$ \emph{is antidominant} if
\[
\overset{\vee}{\alpha}(\lambda)\notin\mathbb{N}\text{ \ for each positive root
}\alpha\in\Sigma^{+}\text{.}%
\]

Suppose $\mathcal{O}_{\lambda}$ is the sheaf of holomorphic sections of a
$G_{0}$-equivariant line bundle. Let $S$ be a $G_{0}$-orbit in $X$ and let $Q$
be the Matsuki dual to $S$. We define\emph{ the vanishing number} $q$
\emph{of} $S$ to be the (complex) codimension of $Q$ in $X$. Suppose $\lambda$
is antidominant and regular. The one of the main results of the Hecht-Taylor
analytic localization theory is that the compactly supported sheaf
cohomologies of the restriction $\mathcal{O}_{\lambda}\mid_{S}$ of
$\mathcal{O}_{\lambda}$ to $S$ vanish, except in degree $q$, in which case
\[
H_{\text{c}}^{q}(S,\mathcal{O}_{\lambda}\mid_{S})
\]
is the minimal globalization of a corresponding standard Beilinson-Bernstein
module (we will describe this module in the following section). In particular,
one knows $H_{\text{c}}^{q}(S,\mathcal{O}_{\lambda}\mid_{S})$ has a unique
irreducible submodule. In general (for any parameter $\lambda$) one knows
\cite{bratten3} that the sheaf cohomology groups
\[
H_{\text{c}}^{p}(S,\mathcal{O}_{\lambda}\mid_{S})
\]
vanish for $p<q$ and that in the nonzero cases these cohomology groups are
minimal globalizations of the sheaf cohomology groups of an associated
standard Harish-Chandra sheaf. 

In general terms, we now consider a simple geometric procedure which can be
used in the context of the Hecht-Taylor realization of minimal globalizations,
to study representations. We first remark that the sheaves $\mathcal{O}%
_{\lambda}$ are examples of what is referred to in the Hecht-Taylor
development as DNF (stands for dual nuclear Fr\'{e}chet) sheaves of analytic
$G_{0}$-modules (we will not need the formalism of DNF sheaves of analytic
$G_{0}$-modules in our study however we would simply like to mention the
general criteria used to establish the following results). We want to remark
that, in the case of the global sheaf cohomology on $X$, the sheaf cohomology
groups
\[
H^{p}(X,\mathcal{O}_{\lambda})
\]
are finite-dimensional and were originally studied by Bott in \cite{bott}. Now
suppose $L\subseteq$ $X$ is a locally closed $G_{0}$-invariant subspace and
let
\[
\left(  \mathcal{O}_{\lambda}\mid_{L}\right)  ^{X}%
\]
denote the extension by zero to $X$ of the restriction of $\mathcal{O}%
_{\lambda}$ to $L$. Then there is a natural isomorphism of functors
\[
H^{p}(X,\left(  \mathcal{O}_{\lambda}\mid_{L}\right)  ^{X})\cong H_{\text{c}%
}^{p}(L,\mathcal{O}_{\lambda}\mid_{L})
\]
and it follows from the results of Hecht and Taylor (at least for $\lambda$
regular: in the singular case one can prove this by a tensoring argument as in
\cite{bratten3}) that the sheaf cohomology groups
\[
H_{\text{c}}^{p}(L,\mathcal{O}_{\lambda}\mid_{L})
\]
groups are minimal globalizations of Harish-Chandra modules. Let $W\subseteq
L$ be an open, $G_{0}$-invariant subspace and let $C=L-W$. Then we have the
following short exact sequence of DNF sheaves:of analytic $G_{0}$-modules
\[
0\rightarrow\left(  \mathcal{O}_{\lambda}\mid_{W}\right)  ^{X}\rightarrow
\left(  \mathcal{O}_{\lambda}\mid_{L}\right)  ^{X}\rightarrow\left(
\mathcal{O}_{\lambda}\mid_{C}\right)  ^{X}\rightarrow0\text{.}%
\]
Therefore the corresponding long exact sequence in cohomology%

\[
\cdots\rightarrow H_{\text{c}}^{p}(W,\mathcal{O}_{\lambda}\mid_{W})\rightarrow
H_{\text{c}}^{p}(L,\mathcal{O}_{\lambda}\mid_{L})\rightarrow H_{\text{c}}%
^{p}(C,\mathcal{O}_{\lambda}\mid_{C})\rightarrow H_{\text{c}}^{p+1}%
(W,\mathcal{O}_{\lambda}\mid_{W})\rightarrow\cdots
\]
is a sequence of minimal globalizations with continuous $G_{0}$-morphisms.

Return to the case where $S$ is a $G_{0}$-orbit and let $U$ be the smallest
$G_{0}$-invariant open set that contains $S$. Then the compactly supported
sheaf cohomology groups
\[
H_{\text{c}}^{p}(U,\mathcal{O}_{\lambda}\mid_{U})
\]
are minimal globalizations of Harish-Chandra modules. Recall that $q$ is the
vanishing number of $S$. We will now show that there is a natural embedding of
$H_{\text{c}}^{q}(U,\mathcal{O}_{\lambda}\mid_{U})$ in $H_{\text{c}}%
^{q}(S,\mathcal{O}_{\lambda}\mid_{S})$. Let $W=U-S$. Then $W$ is open and we
have the following short exact sequence of sheaves on $X$:
\[
0\rightarrow\left(  \mathcal{O}_{\lambda}\mid_{W}\right)  ^{X}\rightarrow
\left(  \mathcal{O}_{\lambda}\mid_{U}\right)  ^{X}\rightarrow\left(
\mathcal{O}_{\lambda}\mid_{S}\right)  ^{X}\rightarrow0.
\]
This sequence will induce a sequence of continuous morphisms of minimal
globalizations when we apply the long exact sequence of sheaf cohomology. To
prove we have an inclusion in grade $q$ we use the vanishing on $S$, and the
following Lemma.

\begin{lemma}
Maintain the previously defined notations, Then
\[
H_{\text{c}}^{p}(W,\mathcal{O}_{\lambda}\mid_{W})=0
\]
for $p\leq q$
\end{lemma}

\begin{proof}
First observe that since $Q$ is open in $\overline{Q}$ then for each
$Q_{0}\subseteq$ $\overline{Q}$ such that $Q_{0}\neq Q$ then the codimension
of $Q_{0}$ is strictly bigger than the codimension of $Q$, so that the
vanishing numbers for $G_{0}$-orbits in $W$ are at least $q+1$. Suppose
$O\subseteq W$ is a $G_{0}$-invariant open subset. We define the length of $O$
to be the number of $G_{0}$-orbits contained in $O$. We show that the
announced vanishing result holds for every $G_{0}$-invariant open subset of
$W$ by an induction on length. When $O$ has length one then it is an open
$G_{0}$-orbit so the result holds since it has vanishing number at least
$q+1$. In general let $S_{0}\subseteq O$ be a $G_{0}$-orbit of minimal
dimension. Since $S_{0}$ is open and dense in its closure it follows that any
$G_{0}$-orbit in the closure different from $S_{0}$ has strictly smaller
dimension. Thus
\[
\overline{S_{0}}\cap O=S_{0}%
\]
and $O_{0}=O-S_{0}$ is an open $G_{0}$-invariant of shorter length. Thus the
result follows by induction, using the long exact sequence in cohomology
applied to the short exact sequence of sheaves:
\[
0\rightarrow\left(  \mathcal{O}_{\lambda}\mid_{O_{0}}\right)  ^{X}%
\rightarrow\left(  \mathcal{O}_{\lambda}\mid_{O}\right)  ^{X}\rightarrow
\left(  \mathcal{O}_{\lambda}\mid_{S_{0}}\right)  ^{X}\rightarrow0.
\]

\end{proof}

\begin{corollary}
Let $\mathcal{O}_{\lambda}$ be the sheaf of sections of a $G_{0}$-equivariant
holomorphic line bundle over $X$. Suppose $S\subseteq X$ is $G_{0}$-orbit with
associated vanishing number $q$ and let $U$ be the smallest $G_{0}$-invariant
open submanifold that contains $S$. Then there is a natural inclusion of
analytic $G_{0}$-modules
\[
H_{\text{c}}^{q}(U,\mathcal{O}_{\lambda}\mid_{U})\rightarrow H_{\text{c}}%
^{q}(S,\mathcal{O}_{\lambda}\mid_{S})
\]

\end{corollary}

\section{Localization and standard Beilinson-Bernstein modules}

We begin this section by introducing the sheaves of twisted differential
operators and reviewing some of the necessary theory. For a basic reference on
the algebraic side of the localization theory, we note that the article by
Milicic \cite{milicic} provides a nice overview to the geometric realization
of Harish-Chandra modules given by the Beilinson-Bernstein theory.

Let $U(\mathfrak{g})$ be the enveloping algebra of $\mathfrak{g}$ and let
$Z(\mathfrak{g)}$ be the center of $U(\mathfrak{g})$. \emph{An infinitesimal
character} $\Theta$ is a morphism of algebras (with identity)
\[
\Theta:Z(\mathfrak{g)}\rightarrow\mathbb{C}\text{.}%
\]
We let $U_{\Theta}$ be the quotient of $U(\mathfrak{g})$ by the two-sided
generated from the kernel of $\Theta$ in $Z(\mathfrak{g)}$. Observe that
$U_{\Theta}$ is the algebra that acts naturally on a $\mathfrak{g}$-module
with infinitesimal character $\Theta$.

Let $W$ be the Weyl group of $\mathfrak{h}^{\ast}$. By Harish-Chandra's
classical result, $Z(\mathfrak{g})$ is isomorphic to the Weyl group invariants
in the enveloping algebra of a Cartan subalgebra of $\mathfrak{g}$. It follows
that the infinitesimal characters are naturally parametrized by the $W$-orbits
in $\mathfrak{h}^{\ast}$. For $\lambda\in\mathfrak{h}^{\ast}$ we write
\[
\Theta=W\lambda\text{ and }\lambda\in\Theta
\]
when the $W$ orbit of $\lambda$ parameterizes the infinitesimal character
$\Theta$. It is known that if $\lambda\in\mathfrak{h}^{\ast}$ is integral (or
regular) then $w\lambda$ is integral (or regular) for every $w\in W$. In this
case we also say that the corresponding infinitesimal character is
\emph{integral} (or \emph{regular}). When an infinitesimal character $\Theta$
is integral and regular then there exists a unique $\lambda\in\Theta$ that is
antidominant. We remark that the infinitesimal character of an irreducible
admissible representation is an important invariant and that the
Beilinson-Berstein realization of irreducible Harish-Chandra modules with
infinitesimal character $\Theta$ depends, to some extent, on the choice of
$\lambda\in\Theta$.

At this point we need to distinguish between the algebraic and analytic
structures on $X$. \ Therefore we consider the full flag space $X$ as both an
algebraic variety (with the Zariski topology) and as a complex manifold (with
the analytic topology) according to the context. Since the line bundles,
defined for integral $\lambda\in\mathfrak{h}^{\ast}$ in the previous section,
have a compatible algebraic structure, we can consider the corresponding sheaf
of algebraic sections $\mathcal{O}_{\lambda}^{\text{alg}}$ defined on the
algebraic variety $X$. Associated to the sheaf $\mathcal{O}_{\lambda
}^{\text{alg}}$ is a corresponding twisted sheaf of differential operators
(TDO) $\mathcal{D}_{\lambda}^{\text{alg}}$. We can also consider the
corresponding TDO $\mathcal{D}_{\lambda}$, with holomorphic coefficients and
defined on complex variety $X$. In a natural way, $\mathcal{O}_{\lambda}$ is a
sheaf of modules for $\mathcal{D}_{\lambda}$. When $\mathcal{O}_{\lambda}$ is
the sheaf of holomorphic sections of a $G_{0}$-equivariant line bundle then
$G_{0}$ acts on $\mathcal{D}_{\lambda}$ while $K$ acts compatibly (extending
the $K_{0}$-action) on $\mathcal{O}_{\lambda}^{\text{alg}}$ and $\mathcal{D}%
_{\lambda}^{\text{alg}}$. We want to emphasize that we are using the shifted
parametrization from the previous section. In particular, $\mathcal{O}_{-\rho
}$ is the sheaf of holomorphic functions on $X$ and $\mathcal{D}_{-\rho}$ is
the sheaf of holomorphic differential operators.

Suppose $\lambda\in\Theta$. Then Beilinson and Bernstein show that
\[
\Gamma(X,\mathcal{D}_{\lambda}^{\text{alg}})\cong U_{\Theta}\text{ and }%
H^{p}(X,\mathcal{D}_{\lambda}^{\text{alg}})=0\text{ for }p>0.
\]
Thus the sheaf cohomology groups of a sheaf of $\mathcal{D}_{\lambda
}^{\text{alg}}$-modules are $\mathfrak{g}$-modules with infinitesimal
character $\Theta$. When $\mathcal{F}$ is a sheaf of quasi-coherent
$\mathcal{D}_{\lambda}^{\text{alg}}$-modules and $\lambda$ is antidominant,
Beilinson and Bernstein show
\[
\text{ }H^{p}(X,\mathcal{F})=0\text{ for }p>0\text{.}%
\]
In particular, when $\lambda$ is antidominant, the functor of global sections
is exact on the category of quasi-coherent $\mathcal{D}_{\lambda}^{\text{alg}%
}$-modules. The \emph{localization functor}
\[
\Delta_{\lambda}(M)=\mathcal{D}_{\lambda}^{\text{alg}}\otimes_{U_{\Theta}}(M)
\]
is defined on the category of $U_{\Theta}$-modules and is the left adjoint to
the global sections functor on the category of quasi-coherent $\mathcal{D}%
_{\lambda}^{\text{alg}}$-modules. When $\lambda$ is antidominant and regular
then Beilinson and Bernstein show that the localization functor and the global
sections functor are mutual inverses and determine an equivalence of categories.

Suppose $S\subseteq X$ is a $G_{0}$-orbit with associated vanishing number
$q$. Let $\mathcal{O}_{\lambda}$ be the sheaf of holomorphic sections of a
$G_{0}$-equivariant line bundle. We now consider the geometric construction of
the underlying Harish-Chandra module of the minimal globalization
\[
H_{\text{c}}^{q}(S,\mathcal{O}_{\lambda}\mid_{S})\text{.}%
\]
Let $Q$ denote the $K$-orbit Matsuki dual to $S$ (we equip $Q$ with the
Zariski topology) and consider the $K$-equivariant sheaf $\mathcal{O}%
_{\lambda}^{\text{alg}}$. Let $i:Q\rightarrow X$ denote the inclusion and let
\[
i^{\ast}\left(  \mathcal{O}_{\lambda}^{\text{alg}}\right)  =i^{-1}\left(
\mathcal{O}_{\lambda}^{\text{alg}}\right)  \otimes_{i^{-1}\left(
\mathcal{O}_{X}^{\text{alg}}\right)  }\mathcal{O}_{Q}^{\text{alg}}%
\]
be the inverse image of $\mathcal{O}_{\lambda}^{\text{alg}}$ with respect to
the structure sheaves of the algebraic varieties $Q$ and $X$. Therefore
$i^{\ast}\left(  \mathcal{O}_{\lambda}^{\text{alg}}\right)  $ is the sheaf of
sections of a corresponding $K$-homogeneous algebraic line bundle defined on
$Q$. Let $\mathcal{D}_{Q,\lambda}^{\text{alg}}$ be the sheaf of differential
operators for the locally free sheaf $i^{\ast}\left(  \mathcal{O}_{\lambda
}^{\text{alg}}\right)  $. \ Then there is a corresponding \emph{direct image
functor} $i_{+}$ \emph{in the category of sheaves of (twisted)} $\mathcal{D}%
$-\emph{modules.} In this case, since the morphism $i$ is an affine inclusion
of smooth varieties, the direct image is an exact functor \cite{duality}. The
sheaf
\[
\mathcal{I}(Q,\lambda)=i_{+}i^{\ast}\left(  \mathcal{O}_{\lambda}^{\text{alg}%
}\right)
\]
is a $K$-equivariant sheaf of $\mathcal{D}_{\lambda}^{\text{alg}}$-modules
called \emph{the corresponding standard Harish-Chandra sheaf} \emph{on} $X$.
One knows that $i_{+}i^{\ast}\left(  \mathcal{O}_{\lambda}^{\text{alg}%
}\right)  $ contains a unique (coherent and $K$-invariant) irreducible
subsheaf
\[
\mathcal{J}(Q,\lambda)\subseteq\mathcal{I}(Q,\lambda)\text{.}%
\]
of $\mathcal{D}_{\lambda}^{\text{alg}}$-modules. We note that the notation
being used is a bit ambiguous since the structure of these objects also
depends on the $K_{0}$-action and not just the orbit $Q$ and the integral
parameter $\lambda$. However, in the current context it should be clear how
one construction leads to the other and we feel our approach avoids an overly
complicated notation. The Harish-Chandra module
\[
I(Q,\lambda)=\Gamma(X,i_{+}i^{\ast}\left(  \mathcal{O}_{\lambda}^{\text{alg}%
}\right)  )
\]
is called the \emph{corresponding standard Beilinson-Bernstein module}. In
general, one knows \cite{bratten3} that $I(Q,\lambda)$ is the underlying
Harish-Chandra module of $H_{\text{c}}^{q}(S,\mathcal{O}_{\lambda}\mid_{S})$.
Observe that when $\lambda$ is antidominant and regular then, by the
equivalence of categories it follows that the Harish-Chandra module
\[
J(Q,\lambda)=\Gamma(X,\mathcal{J}(Q,\lambda))\subseteq I(Q,\lambda)
\]
is the unique irreducible submodule of the corresponding standard
Beilinson-Bernstein module. We call $I(Q,\lambda)$ \emph{a classifying module
}if $\lambda$ is antidominant and $J(Q,\lambda)\neq0$. As the name suggests,
the classifying modules are used in Beilinson-Bernstein classification of
irreducible admissible representations. This works perfectly when $G_{0}$ is a
connected, complex reductive group, however, in general, one must enlarge the
class of standard Harish-Chandra sheaves to include all irreducible
representations with the given integral infinitesimal character.

Suppose $\lambda$ is antidominant and regular and let $J(Q,\lambda
)_{\text{min}}$ denote the corresponding minimal globalization Then we have
the following.

\begin{proposition}
Maintain the above notations (in particular we assume $\lambda$ is
antidominant and regular). Let $U$ be the smallest $G_{0}$-invariant open set
that contains $S$. Then there exists a natural inclusion
\[
J(Q,\lambda)_{\text{min}}\subseteq H_{\text{c}}^{q}(U,\mathcal{O}_{\lambda
}\mid_{U})\text{.}%
\]

\end{proposition}

\begin{proof}
Since $I(Q,\lambda)$ is the underlying Harish-Chandra module of $H_{\text{c}%
}^{q}(S,\mathcal{O}_{\lambda}\mid_{S})$ it follows that $J(Q,\lambda
)_{\text{min}}$ is the unique irreducible submodule in $H_{\text{c}}%
^{q}(S,\mathcal{O}_{\lambda}\mid_{S})$. Therefore, to establish the result it
suffices to show that
\[
H_{\text{c}}^{q}(U,\mathcal{O}_{\lambda}\mid_{U})\neq0\text{.}%
\]
We follow the set-up used in Lemma 3.2 and argue by contradiction. Suppose
$H_{\text{c}}^{q}(U,\mathcal{O}_{\lambda}\mid_{U})=0$ and let $W=U-S$. Using
the long exact sequence in cohomology we obtain an inclusion
\[
J(Q,\lambda)_{\text{min}}\subseteq H_{\text{c}}^{q}(S,\mathcal{O}_{\lambda
}\mid_{S})\subseteq H_{\text{c}}^{q+1}(W,\mathcal{O}_{\lambda}\mid
_{W})\text{.}%
\]
Now suppose $S_{1}\subseteq W$ is a $G_{0}$-orbit of minimal dimension and put
$W_{1}=W-S$. Then $S_{1}$ is closed in $W_{1}$, so using the corresponding
long exact sequence in cohomology we obtain the sequence
\[
0\rightarrow H_{\text{c}}^{q+1}(W_{1},\mathcal{O}_{\lambda}\mid_{W_{1}%
})\rightarrow H_{\text{c}}^{q+1}(W,\mathcal{O}_{\lambda}\mid_{W})\rightarrow
H_{\text{c}}^{q+1}(S_{1},\mathcal{O}_{\lambda}\mid_{S_{1}})\rightarrow\cdots.
\]
Since $J(Q,\lambda)_{\text{min}}$ is irreducible either
\[
J(Q,\lambda)_{\text{min}}\subseteq H_{\text{c}}^{q+1}(W_{1},\mathcal{O}%
_{\lambda}\mid_{W_{1}})\text{ \ or \ }J(Q,\lambda)_{\text{min}}\subseteq
H_{\text{c}}^{q+1}(S_{1},\mathcal{O}_{\lambda}\mid_{S_{1}})\text{. }%
\]
However, if $H_{\text{c}}^{q+1}(S_{1},\mathcal{O}_{\lambda}\mid_{S_{1}})\neq0$
then this representation is the minimal globalization of the classifying
module $I(Q_{1},\lambda)$ where $Q_{1}$ is the $K$-orbit Matsuki dual to
$S_{1}$. Thus $J(Q_{1},\lambda)_{\text{min}}$ is the unique irreducible
submodule of $H_{\text{c}}^{q+1}(S_{1},\mathcal{O}_{\lambda}\mid_{S_{1}})$.
Since $Q_{1}\neq Q$ it follows that $J(Q,\lambda)$ is not isomorphic to
$J(Q_{1},\lambda)$. Therefore
\[
J(Q,\lambda)_{\text{min}}\subseteq H_{\text{c}}^{q+1}(W_{1},\mathcal{O}%
_{\lambda}\mid_{W_{1}})\text{.}%
\]
Proceeding in this fashion we would obtain that
\[
J(Q,\lambda)_{\text{min}}\subseteq H_{\text{c}}^{q+1}(O,\mathcal{O}_{\lambda
}\mid_{O})
\]
where $O$ is an open $G_{0}$-orbit. However this is impossible since
$H_{\text{c}}^{q+1}(O,\mathcal{O}_{\lambda}\mid_{O})$ is either zero or an
irreducible minimal globalization that is not isomorphic to $J(Q,\lambda
)_{\text{min}}$.
\end{proof}

\bigskip

The proof of our main result now consists of two steps. The first part is to
characterize the irreducible Harish-Chandra sheaf $\mathcal{J}(Q,\lambda)$
when the associated variety $\overline{Q}$ is smooth, Once we have have that
in hand, it turns out to be fairly straightforward to calculate the analytic
localization of $J(Q,\lambda)_{\text{min}}$ on $G_{0}$-orbits. To finish the
proof we show that the inclusion%
\[
J(Q,\lambda)_{\text{min}}\subseteq H_{\text{c}}^{q}(U,\mathcal{O}_{\lambda
}\mid_{U})
\]
induces an isomorphism between analytic localization of $J(Q,\lambda
)_{\text{min}}$ and the sheaf
\[
\left(  \mathcal{O}_{\lambda}\mid_{U}\right)  ^{X}\text{.}%
\]
We can then recover our main result by the Hecht-Taylor equivalence of derived categories.

Getting on with the first step of our proof, we continue with the previous
notations. Let
\[
j:\overline{Q}\rightarrow X
\]
denote the inclusion and assume $\overline{Q}$ is smooth. We consider the
$K$-equivariant sheaf
\[
j^{\ast}\left(  \mathcal{O}_{\lambda}^{\text{alg}}\right)
\]
defined on $\overline{Q}$. Observe that $j^{\ast}\left(  \mathcal{O}_{\lambda
}^{\text{alg}}\right)  $ is the sheaf of sections of a $K$-equivariant
algebraic line bundle defined on $\overline{Q}$. Let $\mathcal{D}%
_{\overline{Q},\lambda}^{\text{alg}}$ denote the sheaf on $\overline{Q}$, of
(twisted) differential operators of the invertible sheaf $j^{\ast}\left(
\mathcal{O}_{\lambda}^{\text{alg}}\right)  $ and let $j_{+}$ be the
corresponding direct image functor, Thus
\[
j_{+}j^{\ast}\left(  \mathcal{O}_{\lambda}^{\text{alg}}\right)
\]
is a $K$-equivariant sheaf of $\mathcal{D}_{\lambda}^{\text{alg}}$-modules.

\begin{proposition}
Suppose $\mathcal{O}_{\lambda}^{\text{alg}}$ is the sheaf of sections of a
$K$-equivariant algebraic line bundle on $X$. Let $Q\subseteq X$ be a
$K$-orbit and suppose $\mathcal{I}(Q,\lambda)$ is the corresponding standard
Harish-Chandra sheaf. Assume that the associated variety $\overline{Q}$ is
smooth and let $j:\overline{Q}\rightarrow X$ be the inclusion. Then there
exists a natural isomorphism
\[
j_{+}j^{\ast}\left(  \mathcal{O}_{\lambda}^{\text{alg}}\right)  \cong%
\mathcal{J}(Q,\lambda)\text{.}%
\]

\end{proposition}

\begin{proof}
Let
\[
l:Q\rightarrow\overline{Q}%
\]
be the inclusion and recall that $i:Q\rightarrow X$. Since $Q$ is open in
$\overline{Q}$, it is clear that
\[
j^{\ast}\left(  \mathcal{O}_{\lambda}^{\text{alg}}\right)  \mid_{Q}\cong
i^{\ast}\left(  \mathcal{O}_{\lambda}^{\text{alg}}\right)
\]
as sheaves of $K$-equivariant $\mathcal{D}_{Q,\lambda}^{\text{alg}}$-modules.
It also follows that the direct image $l_{+}$ coincides with the direct image
$l_{\bullet}$ in the category of sheaves. By the adjointness property of the
direct image
\[
\text{Hom}\left(  j^{\ast}\left(  \mathcal{O}_{\lambda}^{\text{alg}}\right)
,l_{\ast}i^{\ast}\left(  \mathcal{O}_{\lambda}^{\text{alg}}\right)  \right)
\cong\text{Hom}\left(  j^{\ast}\left(  \mathcal{O}_{\lambda}^{\text{alg}%
}\right)  \mid_{Q},i^{\ast}\left(  \mathcal{O}_{\lambda}^{\text{alg}}\right)
\right)
\]
so the isomorphism above determines a nonzero morphism
\[
j^{\ast}\left(  \mathcal{O}_{\lambda}^{\text{alg}}\right)  \rightarrow
l_{\ast}i^{\ast}\left(  \mathcal{O}_{\lambda}^{\text{alg}}\right)
\]
of $K$-equivariant $\mathcal{D}_{\overline{Q},\lambda}^{\text{alg}}$-modules.
Since $\overline{Q}$ is a closed, smooth subvariety of $X$, Kashiwara's
equivalence of categories says that the direct image $j_{+}$ establishes an
equivalence between the category of coherent $\mathcal{D}_{\overline
{Q},\lambda}^{\text{alg}}$-modules and the category of coherent $\mathcal{D}%
_{\lambda}^{\text{alg}}$-modules with support on $\overline{Q}$. Thus we have
a nonzero morphism
\[
j_{+}j^{\ast}\left(  \mathcal{O}_{\lambda}^{\text{alg}}\right)  \rightarrow
j_{+}l_{\ast}i^{\ast}\left(  \mathcal{O}_{\lambda}^{\text{alg}}\right)  \cong
i_{+}i^{\ast}\left(  \mathcal{O}_{\lambda}^{\text{alg}}\right)  =\mathcal{I}%
(Q,\lambda)
\]
this last isomorphism since $i=j\circ l$. Now we simply observe that $j^{\ast
}\left(  \mathcal{O}_{\lambda}^{\text{alg}}\right)  $ is an irreducible
$\mathcal{D}_{\overline{Q},\lambda}^{\text{alg}}$-module so that $j_{+}%
j^{\ast}\left(  \mathcal{O}_{\lambda}^{\text{alg}}\right)  $ is also
irreducible (once again by Kashiwara's equivalence), Since the morphism we
have defined is nonzero and since $i_{+}i^{\ast}\left(  \mathcal{O}_{\lambda
}^{\text{alg}}\right)  $ has a unique irreducible coherent subsheaf, the
proposition is proven.
\end{proof}

\section{Analytic localization and comparison}

We are now ready to introduce the analytic localization. The Hecht-Taylor
version of the localization functor is built around the topology of the
minimal globalization. One the one hand, Hecht and Taylor consider topological
$U_{\Theta}$-modules that have a DNF (dual nuclear Fr\'{e}chet) topology,
where morphims are continuous morphisms of modules, and on the other hand they
define the concept of a DNF sheaf of $\mathcal{D}_{\lambda}$-modules with an
accompanying concept of continuous morphisms of DNF sheaves of modules. For
$\lambda\in\Theta$, the topological localization
\[
\Delta_{\lambda}^{\text{an}}(M)=\mathcal{D}_{\lambda}\widehat{\otimes
}_{U_{\Theta}}(M)
\]
does not have very interesting results, but since free resolutions of DNF
modules are complexes of DNF modules, using these sorts of resolutions one can
define a derived functor $L\Delta_{\lambda}^{\text{an}}$. In particular, the
analytic localization takes complexes of DNF $U_{\Theta}$-modules to complexes
of DNF sheaves of $\mathcal{D}_{\lambda}$-modules (by applying $\Delta
_{\lambda}^{\text{an}}$ to the corresponding free resolutions). On the other
side of the equation, by using Cech resolutions, Hecht and Taylor show there
are enough injectives within the category of DNF sheaves of $\mathcal{D}%
_{\lambda}$-modules. Thus, by applying the global sections to injective
resolutions, one can define a derived global sections functor on complexes of
DNF sheaves of $\mathcal{D}_{\lambda}$-modules . The result is then a complex
of DNF $U_{\Theta}$-modules. On appropriately defined derived categories, for
$\lambda$ regular, it is not hard to show the derived functors $L\Delta
_{\lambda}^{\text{an}}$ and $R\Gamma$ are mutual inverses.

In general one does not know about the homology groups of the analytic
localization of a complex of DNF $U_{\Theta}$-modules: these homology groups
may very well not be DNF sheaves (although they will be $\mathcal{D}_{\lambda
}$-modules). However the homology groups of the analytic localization of a
minimal globalization $M$ (any DNF\ $U_{\Theta}$-module can be thought of as a
complex which is zero in all degrees except zero ) turn out to be DNF sheaves
of $\mathcal{D}_{\lambda}$-modules of a very special sort. In order to explain
this, we introduce the concept of \emph{the geometric fiber }of a sheaf of
$\mathcal{O}_{X}$-modules. In particular, if $\mathcal{F}$ is a sheaf of
$\mathcal{O}_{X}$-modules and $x\in X$ then we define \emph{the geometric
fiber }$T_{x}(\mathcal{F})$ \emph{of} $\mathcal{F}$ \emph{at} $x$ by
\[
T_{x}(\mathcal{F})=\mathbb{C\otimes}_{\mathcal{O}_{X,x}}\mathcal{F}_{x}%
\]
where $\mathcal{O}_{X,x}$ and $\mathcal{F}_{x}$ denote the corresponding
stalks of these sheaves at $x$ and where $\mathcal{O}_{X,x}$ acts on
$\mathbb{C}$ by evaluation at $x$. Then, letting $G_{0}\left[  x\right]  $ is
the stabilizer of $x$ in $G_{0}$, Hecht and Taylor show (for $\lambda$
regular) that the geometric fiber
\[
T_{x}\left(  L_{p}\Delta_{\lambda}^{\text{an}}(M)\right)
\]
of the $p$-th homology group $L_{p}\Delta_{\lambda}^{\text{an}}(M)$ of the
analytic localization of the minimal globalization $M$ is a finite-dimensional
(continuous) $(\mathfrak{h}_{x}$, $G_{0}\left[  x\right]  )$-module, where
$\mathfrak{h}_{x}$ acts by the evaluation of $\lambda+\rho\in\mathfrak{h}%
^{\ast}$ to $x$, and that the restriction of $L_{p}\Delta_{\lambda}%
^{\text{an}}(M)$ \ to the $G_{0}$-orbit $S$ of $x$ is the sheaf of (restricted
holomorphic) sections of the corresponding homogeneous vector bundle over $S$.

When $\Theta$ is a regular and $\lambda\in\Theta$ is antidominant then
\emph{the comparison theorem} \cite{hecht2} provides a way to understand the
analytic localization $L\Delta_{\lambda}^{\text{an}}(M)$ of a minimal
globalization $M$, with infinitesimal character $\Theta$, assuming one
understands the (derived) geometric fibers of the localization
\[
\Delta_{\lambda}(M_{\text{HC}})
\]
of the underlying Harish-Chandra module $M_{\text{HC}}$ of $M$. To explain
this result, we introduce the geometric fiber
\[
T_{x}^{\text{alg}}(\mathcal{F})=\mathbb{C\otimes}_{\mathcal{O}_{X,x}%
^{\text{alg}}}\mathcal{F}_{x}%
\]
of a sheaf of $\mathcal{O}_{X}^{\text{alg}}$-modules $\mathcal{F}$ at $x\in
X$. We note $T_{x}^{\text{alg}}$ defines a left exact functor on (for example)
on the category of quasicoherent $\mathcal{D}_{\lambda}^{\text{alg}}$-modules
and there are corresponding derived functors $L_{p}T_{x}^{\text{alg}}$. When
$M_{\text{HC}}$ is a Harish-Chandra module with infinitesimal character
$\Theta$ and $\lambda\in\Theta$ is antidominant and regular then Beilinson and
Bernstein have shown that the $\mathfrak{h}_{x}$-modules
\[
L_{p}T_{x}^{\text{alg}}\Delta_{\lambda}(M_{\text{HC}})
\]
are finite-dimensional (algebraic) $(\mathfrak{h}_{x}$, $K\left[  x\right]
)$-modules, where $K\left[  x\right]  $ is the stabilizer of $x$ in $K$ and
$\mathfrak{h}_{x}$ acts by the evaluation of $\lambda+\rho\in\mathfrak{h}%
^{\ast}$ to $x$. The comparison theorem says that when $x$ is a special point,
then there exists a natural equivalence between the finite-dimensional
$(\mathfrak{h}_{x}$, $G_{0}\left[  x\right]  )$-modules and the
finite-dimensional $(\mathfrak{h}_{x}$, $K\left[  x\right]  )$-modules (given
by the $(\mathfrak{h}_{x}$, $K_{0}\left[  x\right]  )$-structure) and that
there is a natural isomorphism
\[
L_{p}T_{x}^{\text{alg}}\Delta_{\lambda}(M_{\text{HC}})\cong T_{x}\left(
L_{p}\Delta_{\lambda}^{\text{an}}(M)\right)
\]
of $(\mathfrak{h}_{x}$, $K_{0}\left[  x\right]  )$-modules. We note that a
more general comparison theorem, described in exactly these terms is proved in
\cite[Theorem 7.2]{bratten3}.

We are now ready to prove our main result.

\begin{theorem}
Suppose $\mathcal{O}_{\lambda}$ is the sheaf of holomorphic sections of a
$G_{0}$-equivariant line bundle with regular antidominant parameter
$\lambda\in\mathfrak{h}^{\ast}$. Let $S\subseteq X$ be a $G_{0}$-orbit with
vanishing number $q$ and let $U\supseteq S$ be the smallest $G_{0}$-invariant
open submanifold that contains $S$. Let $Q$ be the $K$-orbit that is Matsuki
dual to $S$ and suppose the associated variety $\overline{Q}$ is smooth. Then
the sheaf cohomology groups
\[
H_{\text{c}}^{p}(U,\mathcal{O}_{\lambda}\mid_{U})
\]
vanish except in degree $q$ in which case $H_{\text{c}}^{q}(U,\mathcal{O}%
_{\lambda}\mid_{U})$ is the unique irreducible submodule of the standard
module $H_{\text{c}}^{q}(S,\mathcal{O}_{\lambda}\mid_{S})$.
\end{theorem}

\begin{proof}
Utilizing the notation from the previous section, we know that $H_{\text{c}%
}^{q}(S,\mathcal{O}_{\lambda}\mid_{S})$ is the minimal globalization of the
standard Beilinson-Bernstein module $I(Q,\lambda)$. Let $J(Q,\lambda)\subseteq
I(Q,\lambda)$ be the corresponding unique irreducible Harish-Chandra submodule
and $J(Q,\lambda)_{\text{min}}$ its minimal globalization. Consider the
complex
\[
R\Gamma\left(  \left(  \mathcal{O}_{\lambda}\mid_{U}\right)  ^{X}\right)
\]
of DNF $U_{\Theta}$-modules and let $J(Q,\lambda)_{\text{min}}\left[
-q\right]  $ denote the complex which has zeros in all grades except $q$ where
we have the module $J(Q,\lambda)_{\text{min}}$. The point of our proof is to
present a nonzero morphism in the derived category
\[
J(Q,\lambda)_{\text{min}}\left[  -q\right]  \rightarrow R\Gamma\left(  \left(
\mathcal{O}_{\lambda}\mid_{U}\right)  ^{X}\right)
\]
such that the induced morphism
\[
L\Delta_{\lambda}^{\text{an}}\left(  J(Q,\lambda)_{\text{min}}\left[
-q\right]  \right)  \rightarrow L\Delta_{\lambda}^{\text{an}}\left(
R\Gamma\left(  \left(  \mathcal{O}_{\lambda}\mid_{U}\right)  ^{X}\right)
\right)  \cong\left(  \mathcal{O}_{\lambda}\mid_{U}\right)  ^{X}\left[
0\right]
\]
is an isomorphism (we include the place holder $\left[  0\right]  $ to
emphasize we think of the sheaf $\left(  \mathcal{O}_{\lambda}\mid_{U}\right)
^{X}$ as a complex concentrated in degree zero). By the equivalence of derived
categories we will thus obtain an isomorphism
\[
R\Gamma\left(  \left(  \mathcal{O}_{\lambda}\mid_{U}\right)  ^{X}\right)
\cong J(Q,\lambda)_{\text{min}}\left[  -q\right]
\]
which is the desired result. 

To present a morphism in the derived category, recall by Proposition 4.1 we
have a natural inclusion
\[
J(Q,\lambda)_{\text{min}}\rightarrow H_{\text{c}}^{q}(U,\mathcal{O}_{\lambda
}\mid_{U}).
\]
Since the sheaf cohomology groups of $\left(  \mathcal{O}_{\lambda}\mid
_{U}\right)  ^{X}$ vanish in degrees smaller that $q$, a standard truncation
argument provides a nonzero morphism in the derived category
\[
H_{\text{c}}^{q}(U,\mathcal{O}_{\lambda}\mid_{U})\left[  -q\right]
\rightarrow R\Gamma\left(  \left(  \mathcal{O}_{\lambda}\mid_{U}\right)
^{X}\right)  .
\]
Composing this morphism with the inclusion gives the desired result.

We now want to show that the cohomology groups of the complex $L\Delta
_{\lambda}^{\text{an}}\left(  J(Q,\lambda)_{\text{min}}\left[  -q\right]
\right)  $ vanish except in degree zero. That is: we want to calculate the
homology groups
\[
L_{p}\Delta_{\lambda}^{\text{an}}\left(  J(Q,\lambda)_{\text{min}}\right)
\]
and see they vanish except in degree $q$. To do this we use the comparison
theorem. So we need to calculate the derived geometric fibers of the sheaf
\[
\Delta_{\lambda}(J(Q,\lambda))=\mathcal{J}(Q,\lambda).
\]
Let $j:\overline{Q}\rightarrow X$ denote the inclusion. Since $\overline{Q}$
is smooth we have
\[
\mathcal{J}(Q,\lambda)\cong j_{+}j^{\ast}\left(  \mathcal{O}_{\lambda
}^{\text{alg}}\right)  .
\]
Thus, calculating the geometric fibers
\[
L_{p}T_{x}^{\text{alg}}\Delta_{\lambda}\left(  J(Q,\lambda\right)  \cong
L_{p}T_{x}^{\text{alg}}\left(  j_{+}j^{\ast}\left(  \mathcal{O}_{\lambda
}^{\text{alg}}\right)  \right)
\]
is a straightforward application of the base change formula for the direct
image in the category of TDOs. In particular, since the codimension of
$\overline{Q}$ in $X$ is $q$. It follows that, for each $x\in X$ that
\[
L_{p}T_{x}^{\text{alg}}\Delta_{\lambda}\left(  J(Q,\lambda\right)  =0\text{ if
}p\neq q\text{.}%
\]
When $x\in$ $\overline{Q}$ then
\[
L_{q}T_{x}^{\text{alg}}\Delta_{\lambda}\left(  J(Q,\lambda\right)  \cong
L_{q}T_{x}^{\text{alg}}\left(  j_{+}j^{\ast}\left(  \mathcal{O}_{\lambda
}^{\text{alg}}\right)  \right)  \cong T_{x}^{\text{alg}}\left(  \mathcal{O}%
_{\lambda}^{\text{alg}}\right)
\]
and when $x\notin\overline{Q}$ then
\[
L_{q}T_{x}^{\text{alg}}\Delta_{\lambda}\left(  J(Q,\lambda\right)  =0.
\]
In particular, if we let
\[
V=\Gamma(X,\mathcal{O}_{\lambda}^{\text{alg}})
\]
be the corresponding irreducible finite-dimensional $(\mathfrak{g},K)$-module
then for each special point $x\in\overline{Q}$ we have
\[
L_{q}T_{x}^{\text{alg}}\Delta_{\lambda}\left(  J(Q,\lambda\right)  \cong
V/\mathfrak{n}_{x}V
\]
as $(\mathfrak{h}_{x},K\left[  x\right]  )$-module.

By the comparison theorem, it follows that the homology groups
\[
L_{p}\Delta_{\lambda}^{\text{an}}\left(  J(Q,\lambda)_{\text{min}}\right)
\]
vanish except when $p=q$ and this, in turn implies that the complex
$L\Delta_{\lambda}^{\text{an}}\left(  J(Q,\lambda)_{\text{min}}\left[
-q\right]  \right)  $ is quasi isomorphic to the complex
\[
L_{q}\Delta_{\lambda}^{\text{an}}\left(  J(Q,\lambda)_{\text{min}}\right)
\left[  0\right]
\]
which is nonzero only in degree $0$. Hence the nonzero morphism
\[
L\Delta_{\lambda}^{\text{an}}\left(  J(Q,\lambda)_{\text{min}}\left[
-q\right]  \right)  \rightarrow\left(  \mathcal{O}_{\lambda}\mid_{U}\right)
^{X}\left[  0\right]
\]
in the derived category reduces to a nonzero morphism
\[
L_{q}\Delta_{\lambda}^{\text{an}}\left(  J(Q,\lambda)_{\text{min}}\right)
\rightarrow\left(  \mathcal{O}_{\lambda}\mid_{U}\right)  ^{X}%
\]
of $G_{0}$-equivariant DNF sheaves of $\mathcal{D}_{\lambda}$-modules. We will
prove that this morphism is an isomorphism. In particular, if $O$ is a $G_{0}%
$-orbit in $U$ then since $L_{q}\Delta_{\lambda}^{\text{an}}\left(
J(Q,\lambda)_{\text{min}}\right)  \mid_{O}$ and $\mathcal{O}_{\lambda}\mid
_{O}$ are both induced equivariant sheaves it follows that the restricted
morphism
\[
L_{q}\Delta_{\lambda}^{\text{an}}\left(  J(Q,\lambda)_{\text{min}}\right)
\mid_{O}\rightarrow\mathcal{O}_{\lambda}\mid_{O}%
\]
is either an isomorphism or zero. We remark that these limited possibilities
for the restricted morphism can also be deduced from the fact that we have a
morphism of $\left(  \mathcal{D}_{\lambda}\mid_{O}\right)  $-modules and both
objects are locally free rank one sheaves of $\left(  \mathcal{O}_{X}\mid
_{O}\right)  $-modules. Indeed, if we knew a priori that $L_{q}\Delta
_{\lambda}^{\text{an}}\left(  J(Q,\lambda)_{\text{min}}\right)  $ was a
locally free sheaf of $\mathcal{O}_{U}$-modules it would follow immediately
from standard $\mathcal{D}$-module theory \cite[1.4.10]{hotta} that a nonzero
morphism of $\mathcal{D}_{\lambda}$-modules would be an isomorphism. 

Define
\[
W=\left\{  x\in U:\text{the induced morphism }L_{q}\Delta_{\lambda}%
^{\text{an}}\left(  J(Q,\lambda)_{\text{min}}\right)  _{x}\rightarrow\left(
\mathcal{O}_{\lambda}\right)  _{x}\text{ is nonzero }\right\}  \text{.}%
\]
We will show that $W$ is an open set that contains $S$. Since $G_{0}$ acts on
$W$ and since $U$ is the smallest $G_{0}$-invariant open set that contains $S$
it will follow from our previous remarks that the morphism in question is an
isomorphism. 

Consider the composition
\[
L_{q}\Delta_{\lambda}^{\text{an}}\left(  J(Q,\lambda)_{\text{min}}\right)
\rightarrow\left(  \mathcal{O}_{\lambda}\mid_{U}\right)  ^{X}\rightarrow
\left(  \mathcal{O}_{\lambda}\mid_{S}\right)  ^{X}%
\]
where the second morphism is the canonical one. Since these morphisms induce
the nonzero composition
\[
J(Q,\lambda)_{\text{min}}\rightarrow H_{\text{c}}^{q}(U,\mathcal{O}_{\lambda
}\mid_{U})\rightarrow H_{\text{c}}^{q}(S,\mathcal{O}_{\lambda}\mid_{S})
\]
it follows that the restricted morphism
\[
L_{q}\Delta_{\lambda}^{\text{an}}\left(  J(Q,\lambda)_{\text{min}}\right)
\mid_{S}\rightarrow\mathcal{O}_{\lambda}\mid_{S}%
\]
is an isomorphism and $S\subseteq W$. To show $W$ is open suppose $x\in W$.
Since $\mathcal{O}_{\lambda}$ is a locally free rank one sheaf of
$\mathcal{O}_{X}$-modules there is a local section $\sigma$ of $\mathcal{O}%
_{\lambda}$, defined on a neighborhood of $x$ such that every local section
has the form $f\sigma$ where $f$ is a holomorphic function. Since the induced
morphism on the geometric fibre
\[
T_{x}\left(  L_{q}\Delta_{\lambda}^{\text{an}}\left(  J(Q,\lambda
)_{\text{min}}\right)  \right)  \rightarrow T_{x}\left(  \mathcal{O}_{\lambda
}\right)
\]
is nonzero it follows that for some open set $W_{1}$, that contains $x$, there
is a holomorphic function $f$, defined on $W_{1}$ such that $f(x)\neq0$, and a
local section in $\Gamma\left(  W_{1},L_{q}\Delta_{\lambda}^{\text{an}}\left(
J(Q,\lambda)_{\text{min}}\right)  \right)  $ that maps onto $f\sigma$. Thus
\[
W_{2}=\left\{  z\in W_{1}:f(z)\neq0\right\}
\]
is an open set such that $x\in W_{2}\subseteq W$ and we have finished the
proof.   
\end{proof}

\section{Some additional considerations}

\subsection{A tensoring argument}

We maintain the notations from the previous section. In particular, $S$
denotes a $G_{0}$-orbit in $X$ and $Q$ is the $K$-orbit that is Matsuki dual
to $S$. $\mathcal{O}_{\lambda}$ is the sheaf of holomorphic sections of a
$G_{0}$-equivariant line bundle on $X$, and so on. When the parameter
$\lambda\in\mathfrak{h}^{\ast}$ is antidominant then it may be the case that
the Harish-Chandra module
\[
J(Q,\lambda)=\Gamma(X,\mathcal{J}(Q,\lambda))
\]
is zero. However, when $J(Q,\lambda)\neq0$ then it is the unique irreducible
submodule of the standard Beilinson-Bernstein module $I(Q,\lambda)$. When
$\lambda$ is antidominant and $J(Q,\lambda)\neq0$ we will refer to
$I(Q,\lambda)$ (as well as its minimal globalization $H_{\text{c}}%
^{q}(S,\mathcal{O}_{\lambda}\mid_{S})$) as a \emph{classifying module}. Let
$U$ be the smallest $G_{0}$-invariant open submanifold that contains $S$.
Under the assumption that the associated variety $\overline{Q}$ is smooth and
$\lambda$ is antidominant we can give the following tensoring argument that
shows $H_{\text{c}}^{q}(U,\mathcal{O}_{\lambda}\mid_{U})$ is the minimal
globalization of $J(Q,\lambda)$. Hence when $H_{\text{c}}^{q}(S,\mathcal{O}%
_{\lambda}\mid_{S})$ is a classifying module it follows that $H_{\text{c}}%
^{q}(U,\mathcal{O}_{\lambda}\mid_{U})$ is the unique irreducible submodule.

\begin{lemma}
Maintain the previous notations and assume $\lambda\in\mathfrak{h}^{\ast}$ is
antidominant. Suppose that the associated variety $\overline{Q}$ is smooth.
Then the sheaf cohomology groups
\[
H_{\text{c}}^{p}(U,\mathcal{O}_{\lambda}\mid_{U})
\]
vanish except in degree $q$ in which case $H_{\text{c}}^{q}(U,\mathcal{O}%
_{\lambda}\mid_{U})$ is the minimal globalization of $J(Q,\lambda)$.
\end{lemma}

\begin{proof}
The proof is basically the same as (but simpler than) the proof in
\cite[Theorem 9.4]{bratten3} with the slight difference that we need to use
the description of the irreducible Harish-Chandra sheaf $\mathcal{J}%
(Q,\lambda)$ from Proposition 4.2 instead of the description for
$\mathcal{I}(Q,\lambda)$. We sketch some details to help the reader adapt the
notation here to the notation in \cite[Section 9]{bratten3}. From the theory
of highest weight modules, one knows there is an irreducible
finite-dimensional $G_{0}$-module $F^{\mu}$ which is irreducible as a
$\mathfrak{g}$-module and has a highest weight $\mu\in\mathfrak{h}^{\ast}$
sufficiently dominant that $\lambda-\mu$ is antidominant and regular. Observe
that $\mathcal{O}_{\lambda-\mu}$ is the sheaf of holomorphic sections of a
$G_{0}$-equivariant line bundle. Let $\Theta$ be the infinitesimal character
\[
\Theta=W\cdot\lambda\text{.}%
\]
If $M$ is a $\mathfrak{g}$-module (or if $\mathcal{M}$ is a sheaf of
$\mathfrak{g}$-modules) we let $M_{\Theta}$ (respectively $\mathcal{M}%
_{\Theta}$) denote the corresponding $Z(\mathfrak{g})$-eigenspace. Then, as in
the proof of \cite[Theorem 9.4]{bratten3}, we have the following natural
isomorphisms:
\[
(i)\text{ \ }\left(  \mathcal{O}_{\lambda-\mu}\mid_{U}\otimes F^{\mu}\right)
_{\Theta}\cong\mathcal{O}_{\lambda}\mid_{U}\text{ \ and}%
\]
\[
(ii)\ \ \left(  \mathcal{J}(Q,\lambda-\mu)\otimes F^{\mu}\right)  _{\Theta
}\cong\mathcal{J}(Q,\lambda)\text{.}%
\]
Taking sheaf cohomology, in the first case we obtain
\[
H_{\text{c}}^{p}(U,\mathcal{O}_{\lambda}\mid_{U})\cong\left(  H_{\text{c}}%
^{p}(U,\mathcal{O}_{\lambda-\mu}\mid_{U})\otimes F^{\mu}\right)  _{\Theta}%
\]
which implies that the compactly supported sheaf cohomology groups
$H_{\text{c}}^{p}(U,\mathcal{O}_{\lambda}\mid_{U})$ vanish except when $p=q$
in which case $H_{\text{c}}^{q}(U,\mathcal{O}_{\lambda}\mid_{U})$ is the
minimal globalization of a Harish-Chandra module. To see which Harish-Chandra
module, we begin with the natural isomorphism from the previous section
\[
\left(  J_{\lambda-\mu}\right)  _{\text{min}}\cong H_{\text{c}}^{q}%
(U,\mathcal{O}_{\lambda-\mu}\mid_{U}).
\]
Therefore we obtain the isomorphism
\[
\left[  \left(  J_{\lambda-\mu}\otimes F^{\mu}\right)  _{\Theta}\right]
_{\text{min}}\cong\left(  H_{\text{c}}^{q}(U,\mathcal{O}_{\lambda-\mu}\mid
_{U})\otimes F^{\mu}\right)  _{\Theta}\cong H_{\text{c}}^{q}(U,\mathcal{O}%
_{\lambda}\mid_{U})\text{.}%
\]
Finally, taking global sections for the isomorphism in $(ii)$ we obtain
\[
\left(  J_{\lambda-\mu}\otimes F^{\mu}\right)  _{\Theta}\cong J_{\lambda}%
\]
which completes the proof of the Lemma.
\end{proof}

\begin{corollary}
Let $S\subseteq X$ be a $G_{0}$-orbit with vanishing number $q$ and let
$U\supseteq S$ be the smallest $G_{0}$-invariant open submanifold that
contains $S$. Let $Q$ be the $K$-orbit that is Matsuki dual to $S$ and suppose
the associated variety $\overline{Q}$ is smooth. Let $\mathcal{O}_{\lambda}$
be the sheaf of holomorphic sections of a $G_{0}$-equivariant line bundle and
suppose
\[
H_{\text{c}}^{q}(S,\mathcal{O}_{\lambda}\mid_{S})
\]
is a classifying module. Then the sheaf cohomology groups
\[
H_{\text{c}}^{p}(U,\mathcal{O}_{\lambda}\mid_{U})
\]
vanish except in degree $q$ in which case $H_{\text{c}}^{q}(U,\mathcal{O}%
_{\lambda}\mid_{U})$ is the unique irreducible submodule of the standard
module $H_{\text{c}}^{q}(S,\mathcal{O}_{\lambda}\mid_{S})$.
\end{corollary}

\subsection{Maximal parabolic subgroups of complex reductive groups}

Suppose $G_{0}$ is a connected, complex reductive group. It turns out that the
representation we are studying has a close relationship to the classical
parabolic induction when the parabolic under consideration is maximal. This
allowed us to consider some examples (with the help of D. Vogan and A. Paul)
to see how the representation works when the associated algebraic variety is
singular. In the examples we considered the representation is irreducible only
when the associated algebraic variety is nonsingular.

In particular, let $H_{0}\subseteq G_{0}$ be a $\theta$-stable Cartan subgroup
($\theta$ is the complex conjugation of $G_{0}$ with respect to a compact real
form) and let $B_{0}\supseteq H_{0}$ be a Borel subgroup. We consider a
maximal, proper parabolic subgroup $P_{0}$ of $G_{0}$ that contains $B_{0}$.
These are determined in the following way. Let $W(G_{0})$ be the Weyl group of
$H_{0}$ (we can think of $W(G_{0})$ as the quotient of the normalizer of
$H_{0}$ in $G_{0}$ over $H_{0}$). Then $W(G_{0})$ acts naturally on the set of
Borel subgroups that contain $H_{0}$. As in the introduction, we let $X_{0}$
be the complex flag manifold of Borel subgroups of $G_{0}$ and let
$X_{0}^{\text{c}}$ be the conjugate complex manifold. Then the flag manifold
$X$ of Borel subalgebras of the complexified Lie algebra $\mathfrak{g}$ of
$\mathfrak{g}_{0}$ can be identified with the direct product
\[
X=X_{0}\times X_{0}^{\text{c}}.
\]
We have the two actions of $G_{0}$ on $X$. The diagonal action
\[
g\cdot(x,y)=(gx,gy)
\]
corresponding to the fact that $G_{0}$ is a real group with real Lie algebra
$\mathfrak{g}_{0}$ and the action
\[
g\cdot(x,y)=(gx,\theta(g)y)
\]
that corresponds to the action of $G_{0}=K$ as the complexification of $K_{0}%
$. As before let $B_{0}^{\text{op}}$ be the Borel subgroup opposite to $B_{0}$
(this subgroup corresponds to the longest element in $W(G_{0})$). Then each
$G_{0}$-orbit and each $K$-orbit on $X$ contains exactly one special point of
the form
\[
\left(  w\cdot B_{0},B_{0}^{\text{op}}\right)  \in X_{0}\times X_{0}%
^{\text{c}}%
\]
so we can identify $G_{0}$-orbits and $K$-orbits with elements of $W(G_{0})$.
One knows that $Q_{w_{1}}\subseteq\overline{Q_{w_{2}}}$ if and only
$w_{1}\preceq w_{2}$ in the Bruhat order $\preceq$ on $W(G_{0})$. In
particular, the Bruhat interval
\[
\left[  1,w\right]  =\left\{  u\in W(G_{0}):u\preceq w\right\}
\]
characterizes the $K$-orbits, $Q_{u}$, contained in $\overline{Q_{w}}$ as well
as the $G_{0}$-orbits, $S_{u}$, contained in the the smallest $G_{0}%
$-invariant open submanifold $U_{w}$ that contains $S_{w}$. Let $n$ be the
number of simple reflections in $W(G_{0})$. Observe that the number of $G_{0}%
$-orbits with vanishing number $1$ is exactly $n$ (the closed $G_{0}$-orbit is
the unique orbit with vanishing number $0$).

Let $Y_{0}$ be the generalized complex flag space of $G_{0}$-conjugates to
$P_{0}$ and let $Y_{0}^{\text{c}}$ be the complex manifold conjugate to
$Y_{0}$. Consider the generalized flag space
\[
Y=Y_{0}\times Y_{0}^{\text{c}}%
\]
and let $C$ be the $G_{0}$-orbit of $y=\left(  P_{0},P_{0}\right)  \in Y$.
Then $C$ is closed in $Y$ and the $G_{0}$-orbit of $y$ is a real form in $Y$.
Let
\[
\pi:X\rightarrow Y
\]
denote the equivariant projection. Then $\pi^{-1}(C)$ is a closed $G_{0}%
$-invariant submanifold and, since $P_{0}$ is a maximal, it contains exactly
$n-1$ orbits with vanishing number $1$. In particular, if $L_{0}\subseteq
P_{0}$ is the Levi factor of $P_{0}$ that contains $H_{0}$ then these $n-1$
orbits correspond to the simple reflections of the Weyl group $W(L_{0})$ of
$H_{0}$ in $L_{0}$ and the orbits in $\pi^{-1}(C)$ correspond to the elements
in $W(L_{0})$. Indeed, by intersection with the fiber, these $G_{0}$-orbits
give the $L_{0}$-orbits in the complex flag manifold
\[
X_{y}=\pi^{-1}(\left\{  y\right\}  )
\]
for $L_{0}$. Let $S$ be the remaining $G_{0}$-orbit with vanishing number $1$
and let
\[
U=X-\pi^{-1}(C)\text{.}%
\]
Then $U$ is the smallest $G_{0}$-invariant open set that contains $S$ (to see
this fact, since $U$ is open and contains $S$, is sufficient to check that $S$
is the unique $G_{0}$-orbit that is closed in $U$).

Let $\mathcal{O}_{\lambda}$ be the sheaf of holomorphic sections of a $G_{0}%
$-equivariant line bundle on $X$ and assume $\lambda$ is antidominant and
regular. In a natural way, the sheaf $\mathcal{O}_{\lambda}$ determines a
corresponding sheaf of holomorphic sections $\mathcal{O}_{X_{y},\lambda}$ for
an $L_{0}$-equivariant line bundle defined on $X_{y}$. Let
\[
F=\Gamma(X_{y},\mathcal{O}_{X_{y},\lambda})
\]
be the corresponding irreducible finite-dimensional representation for $L_{0}$
with lowest weight $\lambda+\rho$. In a unique way, this representation
extends to an irreducible representation
\[
\omega:P_{0}\rightarrow GL(F).
\]
We consider the corresponding classical (un-normalized) parabolic induction
$I_{P_{0}}^{G_{0}}(F)$ given by
\[
I_{P_{0}}^{G_{0}}(F)=\left\{  \text{real analytic functions }\varphi
:G_{0}\rightarrow F:\varphi(gp)=\omega(p^{-1})\varphi(g)\right\}  \text{.}%
\]
Then (for example see \cite{bratten2}) we have a natural isomorphism of
$G_{0}$-modules
\[
\Gamma(\pi^{-1}(C),\mathcal{O}_{\lambda})\cong I_{P_{0}}^{G_{0}}(F)\text{
\ and }H^{p}(\pi^{-1}(C),\mathcal{O}_{\lambda})=0\text{ for }p>0
\]
where we obtain the vanishing by the Leray spectral sequence and the fact that
the sheaf cohomology groups of a real analytic vector bundle over a real
analytic manifold vanish in positive degree.

Now consider the short exact sequence of sheaves
\[
0\rightarrow\left(  \mathcal{O}_{\lambda}\mid_{U}\right)  ^{X}\rightarrow
\mathcal{O}_{\lambda}\rightarrow\left(  \mathcal{O}_{\lambda}\mid_{\pi
^{-1}(C)}\right)  ^{X}\rightarrow0.
\]
Thus we have the following short exact sequence of representations
\[
0\rightarrow V\rightarrow I_{P_{0}}^{G_{0}}(F)\rightarrow H_{\text{c}}%
^{1}(U,\mathcal{O}_{\lambda}\mid_{U})\rightarrow0
\]
where $V=\Gamma(X,$ $\mathcal{O}_{\lambda})$ is the corresponding irreducible
finite-dimensional $G_{0}$-module. Thus the minimal globalization
$H_{\text{c}}^{1}(U,\mathcal{O}_{\lambda}\mid_{U})$ is irreducible if and only
if the quotient
\[
I_{P_{0}}^{G_{0}}(F)/V
\]
is irreducible.

Let $Q$ be the $K$-orbit Matsuki dual to $S$. Then one would like to know when
$\overline{Q}$ is smooth. We did the calculation (which is not difficult) for
$GL(n+1,\mathbb{C})$ and it works like this. The Levi factor of $P_{0}$ is
characterized by a partition
\[
n_{1}+n_{2}=n+1
\]
where
\[
L_{0}=GL(n_{1},\mathbb{C})\times GL(n_{2},\mathbb{C})\subseteq
GL(n+1,\mathbb{C})\text{.}%
\]
It turns out that $\overline{Q}$ is smooth if and only $n_{1}$ and $n_{2}$
belong to $\left\{  n,1\right\}  $. Therefore $I_{P_{0}}^{G_{0}}(F)/V$ is
irreducible in this case. At this point we contacted D. Vogan to see what was
known about the composition factors of these principal series (we asked about
the case when $V=\mathbb{C}$ is the trivial $G_{0}$-module). After doing a
calculation he guessed that there are
\[
\min\left\{  n_{1},n_{2}\right\}
\]
composition factors occurring in the representation
\[
I_{P_{0}}^{G_{0}}(\mathbb{C})/\mathbb{C}\text{.}%
\]
Vogan passed this on to Annegret Paul who confirmed the guess for some low
dimensional examples by using a computer program (apparently the group
$GL(6,\mathbb{C})$ is already a difficult calculation for the algorithms that
were used).

Hence, for these examples, the representation $H_{\text{c}}^{1}(U,\mathcal{O}%
_{\lambda}\mid_{U})$ is irreducible if and only if the associated algebraic
variety is smooth.

\section{Serre duality}

Since the resolutions used in the Hecht-Taylor construction of the derived
category of DNF sheaves of $\mathcal{D}_{\lambda}$-modules are Cech
resolutions, perhaps it is worth mentioning that it is not difficult to
establish the validity of Serre duality using these sorts of resolutions
\cite[Section 10]{bratten3}. In particular, let $n$ be the complex dimension
of $X$ and let $\Omega^{n}$ be the canonical bundle on $X$. Thus, for $x\in
X$, the geometric fiber $T_{x}(\Omega^{n})$ of $\Omega^{n}$ at $x$ is given
by
\[
T_{x}(\Omega^{n})=\wedge^{n}\mathfrak{n}_{x}%
\]
as a $G_{0}\left[  x\right]  $-module (recall that $G_{0}\left[  x\right]  $
denotes the stabilizer of $x$ in $G_{0}$). In particular, $\Omega^{n}$ is a
$G_{0}$-equivariant holomorphic line bundle on $X$. Using the unshifted
notation from Section 3 of this paper, suppose $\mathcal{O}\left(  \mu\right)
$ is the sheaf of holomorphic sections of a $G_{0}$-equivariant line bundle on
$X$. Then the sheaf of holomorphic sections of the dual bundle is given by
$\mathcal{O}\left(  -\mu\right)  $ (i.e. the sheaf of sections of the line
bundle associated to the dual geometric fiber). If $U\subseteq X$ is any
$G_{0}$-invariant open submanifold of $X$ then Serre duality gives a natural
isomorphism of topological $G_{0}$-modules
\[
H_{\text{c}}^{p}(U,\mathcal{O}\left(  \mu\right)  \mid_{U})^{\prime}\cong
H^{n-p}\left(  U,\mathcal{O}\left(  -\mu\right)  \otimes\mathcal{O}\left(
\Omega^{n}\right)  \mid_{U}\right)
\]
where $H_{c}^{p}(U,\mathcal{O}\left(  \mu\right)  \mid_{U})^{\prime}$ denotes
the continuous dual of the topological $G_{0}$-module $H_{c}^{p}%
(U,\mathcal{O}\left(  \mu\right)  \mid_{U})$ and $\mathcal{O}\left(
\Omega^{n}\right)  $ is the sheaf of holomorphic sections of the canonical
bundle. In terms of the shifted $\mathcal{D}$-module parameters $\lambda
\in\mathfrak{h}^{\ast}$ we obtain
\[
H_{\text{c}}^{p}(U,\mathcal{O}_{\lambda}\mid_{U})^{\prime}\cong H^{n-p}\left(
U,\mathcal{O}_{-\lambda}\mid_{U}\right)
\]
for each $p$. In particular, the sheaf cohomology groups of a $G_{0}%
$-equivariant holomorphic line bundle on a $G_{0}$-invariant open submanifold
are maximal globalizations of Harish-Chandra modules.

\end{document}